\documentclass[12pt,reqno]{amsart}
\usepackage{amsmath, amsthm, amssymb}

\advance \topmargin by -\headheight
\advance \topmargin by -\headsep

\evensidemargin \oddsidemargin
\newtheorem{theorem}{Theorem}[section]
\newtheorem{lemma}[theorem]{Lemma}

\newtheorem{conjecture}[theorem]{Conjecture}

\theoremstyle{definition}

\theoremstyle{remark}

\numberwithin{equation}{section}

  \def\bff{{\mathbf f}}

\def\bfh{{\mathbf h}}
\def\bfi{{\mathbf i}}
\def\bfj{{\mathbf j}} \def\bfk{{\mathbf k}}

\def\bfu{{\mathbf u}}
\def\bfv{{\mathbf v}}
\def\bfw{{\mathbf w}}
\def\bfx{{\mathbf x}}
\def\bfy{{\mathbf y}}
\def\bfz{{\mathbf z}}

\def\calI{{\mathcal I}}

\def\calN{{\mathcal N}}

\def\calS{{\mathcal S}}
\def\calT{{\mathcal T}}

\def\dbC{{\mathbb C}}\def\dbN{{\mathbb N}}

\def\dbZ{{\mathbb Z}}\def\dbQ{{\mathbb Q}}

\def\gra{{\mathfrak a}}

\def\grf{{\mathfrak f}}
 \def\grg{{\mathfrak g}}
 \def\grh{{\mathfrak h}}

\def\alp{{\alpha}} \def\bfalp{{\boldsymbol \alpha}}

\def\del{{\delta}} \def\Del{{\Delta}}

  \def\Tet{{\Theta}}
\def\bfiota{{\boldsymbol \iota}} 
\def\kap{{\kappa}}

\def\sig{{\sigma}}

\def\ome{{\omega}} 
\def\d{{\partial}}
\def\eps{\varepsilon}

\def\le{\leqslant} \def\ge{\geqslant}

\def\d{{\,{\rm d}}}

\DeclareMathOperator{\card}{card}

\begin{document}
\title[Vinogradov systems with a slice off]{Vinogradov systems with a slice off}
\author[Julia Brandes]{Julia Brandes}
\address{Mathematical Sciences Research Institute, 17 Gauss Way, Berkeley,
CA 94720-5070, USA}

\address{Mathematical Sciences, Chalmers Institute of Technology and University of
Gothenburg, 412 96 G\"oteborg, Sweden}
\email{brjulia@chalmers.se}
\author[Trevor D. Wooley]{Trevor D. Wooley}
\address{School of Mathematics, University of Bristol, University Walk, Clifton,
Bristol BS8 1TW, United Kingdom}
\email{matdw@bristol.ac.uk}
\dedicatory{In memoriam Klaus Friedrich Roth}
\date{}
\begin{abstract} Let $I_{s,k,r}(X)$ denote the number of integral solutions of the
modified Vinogradov system of equations
$$x_1^j+\ldots +x_s^j=y_1^j+\ldots +y_s^j\quad (\text{$1\le j\le k$, $j\ne r$}),$$
with $1\le x_i,y_i\le X$ $(1\le i\le s)$. By exploiting sharp estimates for an auxiliary mean
value, we obtain bounds for $I_{s,k,r}(X)$ for $1\le r\le k-1$. In particular, when
$s,k\in \dbN$ satisfy $k\ge 3$ and $1\le s\le (k^2-1)/2$, we establish the essentially
diagonal behaviour $I_{s,k,1}(X)\ll X^{s+\eps}$.
\end{abstract}
\maketitle

\section{Introduction}
Systems of symmetric diagonal equations are, by orthogonality,
intimately connected with mean values of exponential sums, and consequently find
numerous applications in the analytic theory of numbers. In this paper we consider the
number $I_{s,k,r}(X)$ of integral solutions of the system of equations
\begin{equation}\label{1.1}
    x_1^j+\ldots +x_s^j=y_1^j+\ldots +y_s^j\quad (\text{$1\le j\le k$, $j\ne r$}),
\end{equation}
with $1\le x_i,y_i\le X$ $(1\le i\le s)$. This system is related to that of Vinogradov in which
the equations (\ref{1.1}) are augmented with the additional slice
$$
    x_1^r+\ldots +x_s^r=y_1^r+\ldots +y_s^r,
$$
and may be viewed as a testing ground for progress on systems not of Vinogradov type.
Relatives of such systems have been employed in work on the existence of rational points
on systems of diagonal hypersurfaces as well as cognate paucity problems (see for
example \cite{BP2017, BR2012, BR2015}). The main conjecture for the system (\ref{1.1})
asserts that whenever $r,s,k\in \dbN$, $r<k$ and $\eps>0$, then
\begin{equation}\label{1.2}
    I_{s,k,r}(X)\ll X^{s+\eps}+X^{2s-(k^2+k-2r)/2}.
\end{equation}
Here and throughout, the constants implicit in Vinogradov's notation may depend on $s$,
$k$, and $\eps$. It is an easy exercise to establish a lower bound for $I_{s,k,r}(X)$ that
shows the estimate (\ref{1.2}) to be best possible, save that when $k>2$ one may expect
to be able to take $\eps$ to be zero. Our focus in this memoir is the diagonal regime
$I_{s,k,r}(X)\ll X^{s+\eps}$, and this we address with some level of success in the case
$r=1$.

\begin{theorem}\label{theorem1.1}
    Let $s,k\in \dbN$ satisfy $k\ge 3$ and
    $1\le s\le (k^2-1)/2$. Then for each $\eps>0$, one has $I_{s,k,1}(X)\ll X^{s+\eps}$.
\end{theorem}

In view of the main conjecture (\ref{1.2}), one would expect the conclusion of Theorem
\ref{theorem1.1} to hold in the extended range $1\le s\le (k^2+k-2)/2$. Previous work
already in the literature falls far short of such ambitious assertions. Work of the second
author from the early 1990's shows that $I_{s,k,r}(X)\ll X^{s+\eps}$ only for
$1\le s\le k$ (see \cite[Theorem 1]{Woo1993}). Meanwhile, as a consequence of the
second author's resolution of the main conjecture in the cubic case of Vinogradov's mean
value theorem \cite[Theorem 1.1]{Woo2016}, one has the bound
$I_{s,3,1}(X)\ll X^{s+\eps}$ for $1\le s\le 4$ (see \cite[Theorem 1.3]{Woo2015}). This
conclusion is matched by that of Theorem \ref{theorem1.1} above in the special case
$k=3$. The ideas underlying recent progress on Vinogradov's mean value theorem can,
however, be brought to bear on the problem of estimating $I_{s,k,r}(X)$. Thus, it is a
consequence of the second author's work on nested efficient congruencing
\cite[Corollary 1.2]{Woo2017} that one has $I_{s,k,r}(X)\ll X^{s+\eps}$ for
$1\le s\le k(k-1)/2$. Such a conclusion could also be established through methods related
to those of Bourgain, Demeter and Guth \cite{BDG2016}, though the necessary details
have yet to be elucidated in the published literature. Both the aforementioned estimate
$I_{4,3,1}(X)\ll X^{4+\eps}$, and the new bound reported in Theorem \ref{theorem1.1}
go well beyond this work based on efficient congruencing and $l^2$-decoupling. Indeed,
when $r=1$ we achieve an estimate tantamount to square-root cancellation in a range of
$2s$-th moments extending the interval $1\le s\le k(k-1)/2$ roughly half way to the full
conjectured range $1\le s\le (k^2+k-2)/2$.\par

Our strategy for proving Theorem \ref{theorem1.1} is based on the proof of the estimate
$I_{4,3,1}(X)\ll X^{4+\eps}$ in \cite[Theorem 1.3]{Woo2015}, though it is flexible
enough to deliver estimates for the mean value $I_{s,k,r}(X)$ with $r\ge 1$, as we now
outline. For each integral solution $\bfx,\bfy$ of the system (\ref{1.1}) with
$1\le \bfx,\bfy\le X$, one has the additional equation
\begin{equation}\label{1.3}
    \sum_{i=1}^s(x_i^r-y_i^r)=h,
\end{equation}
for some integer $h$ with $|h|\le sX^r$. We seek to count all such solutions with $h$ thus
constrained. For each integer $z$ with $1\le z\le X$, we find that whenever $\bfx,\bfy,h$
satisfy (\ref{1.1}) and (\ref{1.3}), then one has
\begin{equation}\label{1.4}
    \sum_{i=1}^s(u_i^j-v_i^j)=\ome_jhz^{j-r}\quad (1\le j\le k),
\end{equation}
where $\ome_j$ is $0$ for $1\le j<r$ and $\binom{j}{r}$ for $r\le j\le k$, and in which
we write $u_i=x_i+z$ and $v_i=y_i+z$ $(1\le i\le s)$. If we are able to obtain significant
cancellation in the number of solutions of the system (\ref{1.4}), now with $\bfu,\bfv$
constrained only by the conditions $1\le u_i,v_i\le 2X$ $(1\le i\le s)$, then the overcounting
by $z$ may be reversed to show that there is significant cancellation in the system
(\ref{1.1}) underpinning the mean value $I_{s,k,r}(X)$. This brings us to consider the
number of solutions of the system
\begin{equation}\label{1.5}
    \sum_{i=1}^{2t}h_iz_i^{j-r}=0\quad (r\le j\le k),
\end{equation}
with $|h_i|\le sX^r$ and $1\le z_i\le X$ $(1\le i\le 2t)$. This auxiliary mean value may be
analysed through the use of multiplicative polynomial identities engineered using ideas
related to those employed in \cite{Woo1993}.\par

The reader may be interested to learn the consequences of this strategy when $r$ is
permitted to exceed $1$. The conclusion of Theorem \ref{theorem1.1} is in fact a special
case of a more general result which, for $r\ge 2$, unfortunately fails to deliver diagonal
behaviour.

\begin{theorem}\label{theorem1.2}
    Let $r,s,k\in \dbN$ satisfy $k>r\ge 1$ and
    $$
        1\le s\le \frac{k(k+1)}{2}-\frac{k(k+1)-r(r-1)}{4\kap},
    $$
    where $\kap$ is an integer satisfying $1\le \kap\le (k-r+2)/2$. Then for each $\eps>0$, one has
    $$
        I_{s,k,r}(X)\ll X^{s+(r-1)(1-1/(2\kap))+\eps}.
    $$
\end{theorem}

When $r>1$, although we do not achieve diagonal behaviour, we do improve on the
estimate $I_{s,k,r}(X)\ll X^{s+r+\eps}$ that follows for $1\le s\le k(k+1)/2$ from the
main conjecture in Vinogradov's mean value theorem via the triangle inequality. When
$r>2$, the bound for $I_{s,k,r}(X)$ obtained in the conclusion of Theorem
\ref{theorem1.2} remains weaker than what could be obtained by interpolating between
the aforementioned bounds $I_{s,k,r}(X)\ll X^{s+\eps}$ $(1\le s\le k(k-1)/2)$ and
$I_{s,k,r}(X)\ll X^{s+r+\eps}$ $(1\le s\le k(k+1)/2)$. The former bound is, however, yet
to enter the published literature.\par

In \S2 we speculate concerning what bounds might hold for a class of mean values
associated with the system (\ref{1.5}). In particular, should a suitable analogue of the
main conjecture hold for this auxiliary mean value, then the conclusion of Theorem
\ref{theorem1.2} would be valid with a value of $\kap$ now permitted to be as large as
$$\kap=\left\lfloor \frac{(k-r)(k+r+1)+2}{4}\right\rfloor .$$
We refer the reader to Conjecture \ref{conjecture2.2} below for precise details, and we
note in particular the constraint \eqref{2.3}. When $r=1$ and $k\equiv 0$ or $3$ modulo
$4$, this would conditionally establish the estimate $I_{s,k,1}(X)\ll X^{s+\eps}$ in the
range $1\le s\le (k^2+k-2)/2$, and hence the main conjecture (\ref{1.2}) in full for these
cases. When $r>1$, this conditional result establishes a bound slightly stronger than
$I_{s,k,r}(X)\ll X^{s+r-1}$ when $1\le s\le (k^2+k-4)/2$, which seems quite respectable.
\par

We begin in \S2 by announcing an auxiliary mean value estimate generalising that
associated with the system (\ref{1.5}). This we establish in \S\S3--6, obtaining a
polynomial identity in \S3 of appropriate multiplicative type, establishing a lemma to count
integral points on auxiliary equations in \S4, and classifying solutions according to the
vanishing of certain sets of coefficients in \S5. In \S6 we combine these ideas with a
divisor estimate to complete the proof of this auxiliary estimate. Finally, in \S7, we provide
the details of the argument sketched above which establishes Theorems \ref{theorem1.1}
and \ref{theorem1.2}.\par

Throughout, the letters $r$, $s$ and $k$ will denote positive integers with $r<k$, and
$\eps$ will denote a sufficiently small positive number. We take $X$ to be a large positive
number depending at most on $s$, $k$ and $\eps$. The implicit constants in the notations
of Landau and Vinogradov will depend at most on $s$, $k$, $\eps$, and the coefficients of
fixed polynomials that we introduce. We adopt the following convention concerning the
number $\eps$. Whenever $\eps$ appears in a statement, we assert that the statement
holds for each $\eps>0$. Finally, we employ the non-standard convention that whenever
$G:[0,1)^k\rightarrow \dbC$ is integrable, then
$$
    \oint G(\bfalp)\d\bfalp =\int_{[0,1)^k}G(\bfalp)\d\bfalp .
$$
Here and elsewhere, we use vector notation liberally in a manner that is easily discerned
from the context.\vskip.2cm

\noindent {\bf Acknowledgements.}
Both authors thank the Fields Institute in Toronto for excellent working conditions
and support that made this work possible during the Thematic Program on Unlikely
Intersections, Heights, and Efficient Congruencing. The work of the first author was
supported by the National Science Foundation under Grant No. DMS-1440140 while the
author was in residence at the Mathematical Sciences Research Institute in Berkeley,
California, during the Spring 2017 semester. The second author's work was supported by a
European Research Council Advanced Grant under the European Union's Horizon 2020
research and innovation programme via grant agreement No. 695223.

\section{An auxiliary mean value}
Our focus in this section and those following lies on the
system of equations (\ref{1.5}), since this is intimately connected with the Vinogradov
system missing the slice of degree $r$. Since little additional effort is required to proceed
in wider generality, we establish a conclusion in which the monomials $z^{j-r}$
$(r\le j\le k)$ in (\ref{1.5}) are replaced by independent polynomials $f_j(z)$. We begin in
this section by introducing the notation required to state our main auxiliary result.\par

Let $t$ be a natural number. When $1\le j\le t$, consider a non-zero polynomial
$f_j\in \dbZ[x]$ of degree $k_j$. We say that $\bff=(f_1,\ldots ,f_t)$ is
{\it well-conditioned} when the degrees of the polynomials $f_j$ satisfy the condition
\begin{equation}\label{2.0}
0\le k_t<k_{t-1}<\ldots <k_1,
\end{equation}
and there is no positive integer $z$ for which $f_1(z)=\ldots =f_t(z)=0$.\par

Let $X$ be a positive number sufficiently large in terms of $t$, $\bfk$ and the coefficients
of $f$. We define the exponential sum $\grg(\bfalp;X)$ by putting
$$
    \grg(\bfalp;X)=\sum_{|h|\le X^r}\sum_{1\le z\le X}e\left(h(f_1(z)\alp_1+\ldots
    +f_t(z)\alp_t)\right) .
$$
Finally, we define the mean value
\begin{equation}\label{2.1}
    A_{s,r}(X;\bff)=\oint |\grg(\bfalp;X)|^{2s}\d\bfalp .
\end{equation}
By orthogonality, the mean value $A_{s,r}(X;\bff)$ counts the number of integral solutions
of the system of equations
\begin{equation}\label{2.2}
    \sum_{i=1}^{2s}h_if_j(z_i)=0\quad (1\le j\le t),
\end{equation}
with $|h_i|\le X^r$ and $1\le z_i\le X$ $(1\le i\le 2s)$. The system (\ref{2.2}) plainly
generalises (\ref{1.5}). Our immediate goal is to establish the mean value estimate
recorded in the following theorem.

\begin{theorem}\label{theorem2.1}
    Let $r$, $s$ and $t$ be natural numbers with
    $t\ge 2s-1$. Then whenever $\bff$ is a well-conditioned $t$-tuple of polynomials having
    integral coefficients, one has $A_{s,r}(X;\bff)\ll X^{r(2s-1)+1+\eps}$.
\end{theorem}

Note that when $r=1$, the conclusion of this theorem is tantamount to exhibiting
square-root cancellation in the mean value (\ref{2.1}), so is essentially best possible.
Indeed, even in situations wherein $r>1$, the solutions of (\ref{2.2}) in which
$z_1=z_2=\ldots =z_{2s}$ make a contribution to $A_{s,r}(X; \bff)$ of order
$X\cdot (X^r)^{2s-1}$, and so the conclusion of Theorem \ref{theorem2.1} is again
essentially best possible. Henceforth, we restrict ourselves to the situation described by the
hypotheses of Theorem \ref{theorem2.1}. Thus, we may suppose that $t\ge 2s-1$, and
that $\bff$ is a well-conditioned $t$-tuple of polynomials $f_j\in \dbZ[x]$ with
$\deg(f_j)=k_j\ge 0$.\par

It seems not unreasonable to speculate that the estimate claimed in the statement of
Theorem \ref{theorem2.1} should remain valid when $s$ is significantly larger than
$(t+1)/2$. The total number of choices for the $2s$ pairs of variables $h_i,z_i$ occurring
in the system (\ref{2.2}) is of order $(X^{r+1})^{2s}$. Meanwhile, the $t$ equations
comprising (\ref{2.2}) involve monomials having typical size of order $X^{r+k_j}$
$(1\le j\le t)$. Thus, for large $s$, one should expect that
$$
    A_{s,r}(X;\bff)\ll (X^r)^{2s-t}X^{2s-k_1-\ldots -k_t}.
$$
Keeping in mind the diagonal solutions discussed above, one is led to the following
conjecture.

\begin{conjecture}\label{conjecture2.2}
    Let $r$, $s$ and $t$ be natural numbers, and
    suppose that $\bff$ is a well-conditioned $t$-tuple of polynomials having integral
    coefficients, with $\deg(f_j)=k_j$ $(1\le j\le t)$. Then one has
    $$
        A_{s,r}(X;\bff)\ll X^\eps (X^{r(2s-1)+1}+X^{2s(r+1)-tr-k_1-\ldots -k_t}).
    $$
\end{conjecture}

In the special case in which $t=k-r+1$ and $k_j=j-1$ $(1\le j\le t)$ relevant to the system
(\ref{1.5}), this conjectural bound reads
$$
    A_{s,r}(X;\bff)\ll X^{\eps}(X^{r(2s-1)+1}+X^{2s(r+1)-(k+r)(k-r+1)/2}).
$$
In such circumstances, one finds that
$$
    A_{s,r}(X;\bff)\ll X^{r(2s-1)+1+\eps},
$$
provided that $s$ is an integer satisfying
\begin{equation}\label{2.3}
    4s\le (k-r)(k+r+1)+2.
\end{equation}

We finish this section by remarking that the estimate $A_{s,r}(X;\bff)\ll X^{2rs}$ is
fairly easily established when $t\ge 2s$, a stronger condition than that imposed in
Theorem \ref{theorem2.1}, as we now sketch. We may suppose that $t=2s$ without loss,
and in such circumstances the equations (\ref{2.2}) may be interpreted as a system of
$2s$ linear equations in $2s$ variables $h_i$. There are $O(X^{2s})$ choices for the
variables $z_i$, contributing $O(X^{2s})$ to $A_{s,r}(X;\bff)$ from those solutions with
$\bfh={\mathbf 0}$. Meanwhile, if $\bfh\ne {\mathbf 0}$ one must have
\begin{equation}\label{2.4}
    \det\left( f_j(z_i)\right)_{1\le i,j\le 2s}=0.
\end{equation}
By applying the theory of Schur functions (see Macdonald \cite[Chapter I]{Mac1979}) as
in the proof of \cite[Lemma 1]{PW2002}, one finds that
$$
    \det(f_j(z_i))_{1\le i,j\le 2s}=\Tet(\bfz;\bff)\prod_{1\le i<j\le 2s}(z_i-z_j),
$$
where the polynomial $\Tet(\bfz;\bff)$ is asymptotically definite, meaning that whenever
$z_i$ is sufficiently large for $1\le i\le 2s$, then $|\Tet(\bfz;\bff)|\ge 1$.\par

The contribution to $A_{s,r}(X;\bff)$ arising from the solutions of (\ref{2.2}) with
$z_i=O(1)$, for some index $i$, is $O((X^r)^{2s})$. For if $z_i=O(1)$, then we may fix
$h_i$, and interpret the system as a mean value of exponential sums, applying the
triangle inequality. An application of H\"older's inequality reveals that if such solutions
dominate, then
$$
    A_{s,r}(X;\bff)\ll X^r \oint |\grg(\bfalp;X)|^{2s-1}\d\bfalp \ll X^r A_{s,r}(X;\bff)^{1-1/(2s)},
$$
and the desired conclusion follows. Meanwhile, if $z_i$ is sufficiently large for each index
$i$, then $|\Tet(\bfz; \bff)|$ is strictly positive and hence (\ref{2.4}) can hold only
when $z_i=z_j$ for some indices $i$ and $j$ with $1\le i<j\le 2s$. By symmetry we may
suppose that $i=2s-1$ and $j=2s$, and then we obtain from (\ref{2.2}) the new system of
equations
$$
    \sum_{i=1}^{2s-1}h_i'f_j(z_i)=0\quad (1\le j\le 2s),
$$
with $h_i'=h_i$ $(1\le i\le 2s-2)$ and $h_{2s-1}'=h_{2s-1}+h_{2s}$. This new system is
of similar shape to (\ref{2.2}), and we may apply an obvious inductive argument to bound
the number of its solutions. Here, we keep in mind that given $h_{2s-1}'$, there are
$O(X^r)$ possible choices for $h_{2s-1}$ and $h_{2s}$. Thus we conclude that if this
second class of solutions dominates, then one has
$$
    A_{s,r}(X;\bff)\ll X^r\cdot X^{r(2s-1)}\ll X^{2rs}.
$$
This completes our sketch of the proof that when $t=2s$, the total number of solutions
counted by $A_{s,r}(X;\bff)$ is $O(X^{2rs})$. The reader will likely have no difficulty in
refining this argument to deliver the conclusion of Theorem \ref{theorem2.1} when $t=2s$.

\section{A polynomial identity} The structure of the polynomials $hf_j(z)$ underlying the
mean value $A_{s,r}(X;\bff)$ permits polynomial identities to be constructed of utility in
constraining solutions of the underlying system of equations (\ref{2.2}). In this section we
construct such identities.\par

For the sake of concision, when $n$ is a natural number and $1\le j\le t$, we define the
polynomial $\sig_{j,n}=\sig_{j,n}(\bfz;\bfh)$ by putting
$$
    \sig_{j,n}(\bfz;\bfh)=h_1f_j(z_1)+\ldots +h_nf_j(z_n).
$$

\begin{lemma}\label{lemma3.1}
    Suppose that $n\ge 1$ and that
    $\bff=(f_1,\ldots ,f_{2n+1})$ is a well-conditioned $(2n+1)$-tuple of polynomials having
    integral coefficients. Then there exists a polynomial
    $\Psi_n(\bfw)\in \dbZ[w_1,\ldots ,w_{2n+1}]$, with total degree and coefficients depending
    at most on $n$, $\bfk$ and the coefficients of $\bff$, such that
    \begin{equation}\label{3.1}
        \Psi_n(\sig_{1,n}(\bfz;\bfh),\ldots,\sig_{2n+1,n}(\bfz;\bfh))=0
    \end{equation}
    identically in $\bfz$ and $\bfh$, and yet
    \begin{equation}\label{3.2}
        \Psi_n(\sig_{1,n+1}(\bfz;\bfh),\ldots ,\sig_{2n+1,n+1}(\bfz;\bfh))\ne 0.
    \end{equation}
\end{lemma}

\begin{proof}
    We apply an argument similar to that of \cite[Lemma 1]{Woo1993} based on a consideration of transcendence degrees. Let $K=\dbQ(\sig_{1,n},\ldots ,\sig_{2n+1,n})$. Then $K\subseteq \dbQ(z_1,\ldots, z_n,h_1,\ldots ,h_n)$, so that $K$ has transcendence degree at most $2n$ over $\dbQ$. It follows that the $2n+1$ polynomials $\sig_{1,n}(\bfz;\bfh),\ldots ,\sig_{2n+1,n}(\bfz;\bfh)$ cannot be algebraically independent over $\dbQ$. Consequently, there exists a non-zero polynomial $\Psi_n\in \dbZ[w_1,\ldots ,w_{2n+1}]$ satisfying the property (\ref{3.1}).\par

    It remains now only to confirm that a choice may be made for this non-trivial polynomial $\Psi_n$ in such a manner that property (\ref{3.2}) also holds. In order to establish this claim, we begin by considering any non-zero polynomial $\Psi_n$ of smallest total degree satisfying (\ref{3.1}). Suppose, if possible, that $\Psi_n(\sig_{1,n+1},\ldots ,\sig_{2n+1,n+1})$ is also identically zero. Then the polynomials
    \begin{equation}\label{3.3}
        \frac{\partial}{\partial z_i}\Psi_n(\sig_{1,n+1}(\bfz;\bfh),\ldots,\sig_{2n+1,n+1}(\bfz;\bfh))
    \end{equation}
    and
    \begin{equation}\label{3.4}
        \frac{\partial}{\partial h_i}\Psi_n(\sig_{1,n+1}(\bfz;\bfh),\ldots,\sig_{2n+1,n+1}(\bfz;\bfh))
    \end{equation}
    must also be identically zero for $1\le i\le n+1$. Write
    $$
        u_j=\frac{\partial}{\partial w_j}\Psi_n(w_1,\ldots ,w_{2n+1})\quad (1\le j\le 2n+1),
    $$
    in which we evaluate the right hand side at $w_i=\sig_{i,n+1}(\bfz;\bfh)$
$(1\le i\le 2n+1)$. Then it follows from an application of the chain rule that the vanishing
of the polynomials (\ref{3.3}) and (\ref{3.4}) implies the relations
    \begin{equation}\label{3.5}
        \sum_{j=1}^{2n+1}h_if_j'(z_i)u_j=0\quad (1\le i\le n)
    \end{equation}
    and
    \begin{equation}\label{3.6}
        \sum_{j=1}^{2n+1}f_j(z_i)u_j=0\quad (1\le i\le n+1).
    \end{equation}
    Notice here that we have deliberately omitted the index $i=n+1$ from the relations (\ref{3.5}), since this is superfluous to our needs.\par

    In order to encode the coefficient matrix associated with the system of linear equations in $\bfu$ described by the relations (\ref{3.5}) and (\ref{3.6}), we introduce a block matrix as follows. We define the $n\times (2n+1)$ matrix
    $$
        A_n=(h_if_j'(z_i))_{\substack{\\1\le i\le n\\ 1\le j\le 2n+1}}
    $$
    and the $(n+1)\times (2n+1)$ matrix
    $$
        B_n=(f_j(z_i))_{\substack{1\le i\le n+1\\ 1\le j\le 2n+1}},
    $$
    and then define the $(2n+1)\times (2n+1)$ matrix $D_n$ via the block decomposition
    $$
        D_n=\left(\begin{matrix}A_n\\B_n\end{matrix}\right).
    $$
    We claim that $\det(D_n)$ is not identically zero as a polynomial. The confirmation of this fact we defer to the end of this proof.\par

    With the assumption $\det(D_n)\ne 0$ in hand, one sees that the system of equations (\ref{3.5}) and (\ref{3.6}) has only the trivial solution $\bfu={\mathbf 0}$ over $K$. However, since $\Psi_n(\bfw)$ is a non-constant polynomial, at least one of the derivatives
    $$
        \frac{\partial}{\partial w_j}\Psi_n(w_1,\ldots ,w_{2n+1})\quad (1\le j\le 2n+1)
    $$
    must be non-zero. Suppose that the partial derivative with respect to $w_J$ is non-zero. Then there exists a non-constant polynomial
    $$
        \Psi_n^*(\bfw)=\frac{\partial}{\partial w_J}\Psi_n(w_1,\ldots ,w_{2n+1})
    $$
    having the property that, since $u_J=0$, one has
    $$
        \Psi_n^*(\sig_{1,n}(\bfz;\bfh),\ldots ,\sig_{2n+1,n}(\bfz;\bfh))=0.
    $$
    But the total degree of $\Psi_n^*$ is strictly smaller than that of $\Psi_n$, contradicting our hypothesis that $\Psi_n$ has minimal total degree. We are therefore forced to conclude that the relation (\ref{3.2}) does indeed hold.\par

We now turn to the problem of justifying our assumption that $\det(D_n)\ne 0$. We prove
this assertion for any well-conditioned $(2n+1)$-tuple of polynomials $\bff$ by induction
on $n$. Observe first that when $n=0$, one has $\det(D_0)=f_1(z_1)$. Since $f_1(z)$ is
not identically zero, it follows that $\det(D_0)\ne 0$, confirming the base case of our
inductive hypothesis. We suppose next that $n\ge 1$ and that $\det(D_{n-1})\ne 0$ for
all well-conditioned $(2n-1)$-tuples of polynomials $\bff$, and we seek to show that
$\det(D_n)\ne 0$.\par

    Denote by $\calI$ the set of all $2$-element subsets $\gra=\{a_1,a_2\}$ contained in $\calN=\{1,2,\ldots ,2n+1\}$. When $\gra=\{a_1,a_2\}\in \calI$, we define the matrices
    $$
        A(\gra)=\left( h_if_j'(z_i)\right)_{\substack{2\le i\le n\\ j\in \calN\setminus \gra}}
        \quad \text{and}\quad
        B(\gra)= \left( f_j(z_i)\right)_{\substack{2\le i\le n+1\\ j\in \calN\setminus \gra}}.
    $$
    Equipped with this notation, we define the minors
    $$
        U(\gra)=\det \left( \begin{matrix} h_1f_{a_1}'(z_1)&h_1f_{a_2}'(z_1)\\ f_{a_1}(z_1)&f_{a_2}(z_1)\end{matrix}\right)
        \quad \text{and}\quad
        V(\gra)= \det\left( \begin{matrix}A(\gra)\\ B(\gra)\end{matrix}\right).
    $$
    In this way, we discern that for appropriate choices of $\sig(\gra)\in \{1,-1\}$, the precise nature of which need not detain us, one has
    $$
        \det(D_n)=\sum_{\gra\in \calI}(-1)^{\sig(\gra)}U(\gra)V(\gra).
    $$

    \par By relabelling indices and then applying the inductive hypothesis for the
$(2n-1)$-tuple $(f_3,\ldots ,f_{2n+1})$, it is apparent that $V(\{1,2\})$ is not identically
zero. Moreover, if the leading coefficients of $f_1$ and $f_2$ are $c_1$ and $c_2$,
respectively, then the leading monomial in $U(\{1,2\})$ is
    $$
        (k_1-k_2)c_1c_2h_1z_1^{k_1+k_2-1}\ne 0.
    $$
It follows that $U(\{1,2\})$ is also not identically zero. Since no other minor of the
shape $U(\gra)$, with $\gra \in \calI$ and $\gra\ne \{1,2\}$, has degree $k_1+k_2-1$ or
greater with respect to $z_1$, we deduce that $\det(D_n)$ is not identically zero. This
confirms the inductive hypothesis for the index $n$ and completes the proof of our claim
for all $n$.
\end{proof}

Henceforth, when $n\ge 1$, we fix a choice for the polynomials
$\Psi_n(\bfw)\in \dbZ[w_1,\ldots,w_{2n+1}]$, of minimal total degree, satisfying the
conditions (\ref{3.1}) and (\ref{3.2}). It is useful to extend this definition by taking
$\Psi_0(w)=w$. We may now establish our fundamental polynomial identity.

\begin{lemma}\label{lemma3.2}
    Suppose that $n\ge 0$ and the $(2n+1)$-tuple $\bff=(f_1,\ldots ,f_{2n+1})$ of
polynomials in $\dbZ[x]$ is well-conditioned. Then there exists a non-zero polynomial
$\Phi_n(\bfz;\bfh)\in \dbZ[\bfz,\bfh]$ with the property that
    \begin{equation}\label{3.8}
        \Psi_n(\sig_{1,n+1}(\bfz;\bfh),\ldots ,\sig_{2n+1,n+1}(\bfz;\bfh))=\Phi_n(\bfz;\bfh)
        h_1\cdots h_{n+1}\prod_{1\le i<j\le n+1}(z_i-z_j).
    \end{equation}
\end{lemma}

\begin{proof} In the case $n=0$, the product over $i$ and $j$ on the right hand side of
\eqref{3.8} is empty, and by convention we take this empty product to be $1$. In such
circumstances, we see that $\Psi_0(\sig_{1,1}(z_1;h_1))=h_1f_1(z_1)$, and the
conclusion of the lemma is immediate.\par

Suppose next that $n\ge 1$. Then, when $h_{n+1}=0$, we have
    $$
        \sig_{i,n+1}(\bfz;\bfh)=\sig_{i,n}(\bfz;\bfh)\quad (1\le i\le 2n+1),
    $$
    and thus we deduce from property (\ref{3.1}) of Lemma \ref{lemma3.1} that in this situation, one has
    \begin{equation}\label{3.9}
        \Psi_n(\sig_{1,n+1}(\bfz;\bfh),\ldots ,\sig_{2n+1,n+1}(\bfz;\bfh))=0.
    \end{equation}
    It follows that $h_{n+1}$ divides $\Psi_n(\sig_{1,n+1}(\bfz;\bfh),\ldots ,\sig_{2n+1,n+1}(\bfz;\bfh))$, and by symmetry the
    same holds for $h_1,\ldots ,h_n$. Meanwhile, when $z_n=z_{n+1}$, we have
    $$
        \sig_{i,n+1}(\bfz;\bfh)=\sig_{i,n}(\bfz;h_1,\ldots ,h_{n-1},h_n+h_{n+1}),
    $$
and again we find from property (\ref{3.1}) of Lemma \ref{lemma3.1} that in this special
situation one has (\ref{3.9}). We thus conclude that $z_n-z_{n+1}$ divides the polynomial
$\Psi_n(\sig_{1,n+1}(\bfz;\bfh),\ldots ,\sig_{2n+1,n+1}(\bfz;\bfh))$, and by symmetry
the same holds for $z_i-z_j$ whenever $1\le i<j\le n+1$.\par

    In light of these observations, it is apparent that
    $$
        \Psi_n(\sig_{1,n+1}(\bfz;\bfh),\ldots ,\sig_{2n+1,n+1}(\bfz;\bfh))
    $$
    is divisible by
    $$
        h_1\cdots h_{n+1}\prod_{1\le i<j\le n+1}(z_i-z_j).
    $$
    The quotient of the former polynomial by the latter cannot be zero, since this former polynomial is non-zero, by virtue of property (\ref{3.2}) of Lemma \ref{lemma3.1}. We therefore conclude that a non-zero polynomial $\Phi_n(\bfz;\bfh)\in \dbZ[\bfz,\bfh]$ does indeed exist satisfying (\ref{3.8}). This completes the proof of the lemma.
\end{proof}

It seems quite likely that additional potentially useful structure might be extracted from the polynomial identities provided by Lemma \ref{lemma3.2}. For example, the relation
$$
    (h_1+h_2)(h_1z_1^2+h_2z_2^2)-(h_1z_1+h_2z_2)^2=h_1h_2(z_1-z_2)^2
$$
plays a prominent role in the proof of \cite[Lemma 2.1]{Woo2015}. Meanwhile, writing
$$
    s_j=h_1z_1^j+h_2z_2^j+h_3z_3^j\quad (0\le j\le 4),
$$
one may verify that
\begin{align*}
    (s_1s_4-s_2s_3)^2(s_0s_2-s_1^2)-&(s_0s_4-s_2^2)(s_1s_3-s_2^2)^2\\
    &=h_1h_2h_3(z_1-z_2)^2(z_2-z_3)^2(z_3-z_1)^2F_{6,3}(\bfz;\bfh),
\end{align*}
for a suitable bihomogeneous polynomial $F_{6,3}(\bfz;\bfh)\in \dbZ[\bfz,\bfh]$, of degree $6$ with respect to $\bfz$ and degree $3$ with respect to $\bfh$.

\section{Counting integral solutions pairwise} The polynomial identity furnished by Lemma
\ref{lemma3.2} is of multiplicative type, and particularly powerful when
$\Psi_n(\sig_{1,n+1},\ldots ,\sig_{2n+1,n+1})$ is non-zero for a fixed integral choice of
$\bfz$ and $\bfh$, for then we may exploit elementary estimates for the divisor function.
However, it is possible that the latter quantity
vanishes. This brings us into the domain of the classification of solutions according to the
vanishing or non-vanishing of various intermediate coefficients. We begin with an
elementary lemma concerning polynomials in two variables similar to
\cite[Lemma 2]{Woo1993}, the proof of which we include for the sake of
completeness.

\begin{lemma}\label{lemma4.1}
    Let $\psi \in \dbZ[z,h]$ be a non-trivial polynomial of total degree $d$. Then the
number of integral solutions of the equation $\psi(z,h)=0$ with $|z|\le X$ and
$|h|\le X^r$ is at most $2d(2X^r+1)$.
\end{lemma}

\begin{proof}
We may write $\psi(z,h)=a_d(z)h^d+\ldots +a_1(z)h+a_0(z)$, with $a_i\in \dbZ[z]$ of
degree at most $d$ for $0\le i\le d$. The solutions to be counted are of two types. Firstly,
one has solutions $(z, h)$ with $|z|\le X$ for which $a_i(z)\ne 0$ for some index $i$, and
secondly one has solutions for which $a_i(z)=0$ $(0\le i\le d)$. Given any fixed one of the
(at most) $2X+1$ possible choices of $z$ in a solution of the first type, one finds that $h$
satisfies a non-trivial polynomial equation of degree at most $d$, to which there are at
most $d$ integral solutions. There are consequently at most $d(2X+1)$ solutions of this
first type. On the other hand, whenever $(z, h)$ is a solution of the second type, then
$z$ satisfies some non-trivial polynomial equation $a_i(z)=0$ of degree at most $d$. Since
this equation has at most $d$ integral solutions and there are at most $2X^r+1$ possible
choices for $h$, one has at most $d(2X^r+1)$ solutions of the second type. The
conclusion of the lemma now follows.
\end{proof}

We now announce an initial classification of intermediate coefficients. We define sets
$\calT_{n,m}\subseteq \dbZ[z_1,\ldots ,z_m,h_1,\ldots ,h_m]$ for $0\le m\le n+1$
inductively as follows. First, let $\calT_{n,n+1}$ denote the singleton set containing the
polynomial
\begin{equation}\label{4.1}
    \Psi_n(\sig_{1,n+1}(\bfz;\bfh),\ldots ,\sig_{2n+1,n+1}(\bfz;\bfh)).
\end{equation}
Next, suppose that we have already defined the set $\calT_{n,m+1}$, and consider an
element $\psi\in \calT_{n,m+1}$. We may interpret $\psi$ as a polynomial in $z_{m+1}$
and $h_{m+1}$ with coefficients $\phi(z_1, \dots, z_m; h_1, \dots, h_m)$. We now define
$\calT_{n,m}$ to be the set of all non-zero polynomials
$\phi \in \dbZ[z_1,\dots,z_m,h_1, \dots,h_m]$ occurring as coefficients of elements
$\psi \in \calT_{n,m+1}$ in this way. Note in particular that since the polynomial
(\ref{4.1}) is not identically zero, it is evident that each set $\calT_{n,m}$ is non-empty.\par

This classification of coefficients yields a consequence of Lemma \ref{lemma4.1} of
utility to us in \S6.

\begin{lemma}\label{lemma4.2}
    Let $m$ and $n$ be natural numbers with
    $1\le m\le n\le t$. Suppose that $z_i$ and $h_i$ are fixed integers for $1\le i\le m$ with
    $1\le z_i\le X$ and $|h_i|\le X^r$. Suppose also that there exists $\phi\in \calT_{n,m}$
    having the property that $\phi(z_1,\ldots ,z_m;h_1,\ldots ,h_m)\ne 0$. Then the number
    $N_m(X)$ of integral solutions of the system of equations
    $$
        \psi(z_1,\ldots ,z_{m+1};h_1,\ldots ,h_{m+1})=0\quad (\psi\in \calT_{n,m+1}),
    $$
    with $1\le z_{m+1}\le X$ and $|h_{m+1}|\le X^r$, satisfies $N_m(X)\ll X^r$.
\end{lemma}

\begin{proof}
    It follows from the iterative definition of the sets $\calT_{n,m}$ that any element $\phi\in \calT_{n,m}$ occurs as a coefficient polynomial of an element $\psi \in \calT_{n,m+1}$, when viewed as a polynomial in $h_{m+1}$ and $z_{m+1}$. Fixing any one such
    polynomial $\psi$, we find that for the fixed choice of $z_1,\ldots ,z_m,h_1,\ldots ,h_m$
    presented by the hypotheses of the lemma, the polynomial $\psi(\bfz;\bfh)$ is a non-trivial
    polynomial in $z_{m+1}$, $h_{m+1}$. We therefore conclude from Lemma
    \ref{lemma4.1} that $N_m(X)\ll X^r$. This
completes the proof of the lemma.
\end{proof}

\section{Classification of solutions}
We now address the classification of the set $\calS$ of
all solutions of the system of equations
\begin{equation}\label{5.1}
    \sig_{j,2s}(\bfz;\bfh)=0\quad (1\le j\le 2s-1),
\end{equation}
with $1\le \bfz\le X$ and $|\bfh|\le X^r$. This we execute in two stages. Our discussion is
eased by the use of some non-standard notation. When $(i_1,\ldots ,i_m)$ is an
$m$-tuple of positive integers with $1\le i_1<\ldots <i_m\le 2s$, we abbreviate
$(z_{i_1},\ldots ,z_{i_m})$ to $\bfz_\bfi$ and $(h_{i_1},\ldots ,h_{i_m})$ to $\bfh_\bfi$.\par

In the first stage of our classification, when $0\le n<s$, we say that $(\bfz,\bfh)\in
\calS$ is of type $S_n$ when:
\begin{enumerate}
    \item[(i)]
        for all $(n+1)$-tuples $(i_1,\ldots ,i_{n+1})$ with $1\le i_1<\ldots <i_{n+1}
        \le 2s$, one has
        $$
            \Psi_n(\sig_{1,n+1}(\bfz_\bfi;\bfh_\bfi),\ldots ,\sig_{2n+1,n+1}(\bfz_\bfi;\bfh_\bfi))=0,
        $$
        and
    \item[(ii)]
        for some $n$-tuple $(j_1,\ldots ,j_n)$ with $1\le j_1<\ldots <j_n\le 2s$, one has
        $$
            \Psi_{n-1}(\sig_{1,n}(\bfz_\bfj;\bfh_\bfj),\ldots ,\sig_{2n-1,n}(\bfz_\bfj;\bfh_\bfj))\ne 0.
        $$
\end{enumerate}
Here, we interpret the condition (ii) to be void when $n=0$. Finally, we say that
$(\bfz,\bfh)\in \calS$ is of type $S_s$ when the condition (ii) holds with $n=s$. It follows
that every solution $(\bfz,\bfh)\in \calS$ is of type $S_n$ for some index $n$ with
$0\le n\le s$. We denote the set of all solutions of type $S_n$ by $\calS_n$.\par

In the second stage of our classification, when $1\le n<s$ we subdivide the solutions
$(\bfz,\bfh)\in \calS_n$ as follows. When $0\le m\le n$, we say that a solution
$(\bfz,\bfh)\in \calS_n$ is of type $T_{n,m}$ when condition (ii) holds for the $n$-tuple
$\bfj$, and:
\begin{enumerate}
    \item[(iii)]
        for all $(m+1)$-tuples $(i_1,\ldots ,i_{m+1})$ with $1\le i_1<\ldots <i_{m+1}
        \le 2s$ and $i_l\not\in \{j_1,\ldots ,j_n\}$ $(1\le l\le m+1)$, and for all
        $\psi\in \calT_{n,m+1}$, one has $\psi(\bfz_\bfi;\bfh_\bfi)=0$, and
    \item[(iv)] for some $m$-tuple $(\iota_1,\ldots ,\iota_m)$ with $1\le \iota_1<\ldots
        <\iota_m\le 2s$ and $\iota_l\not\in \{ j_1,\ldots ,j_n\}$ $(1\le l\le m)$, and for some
        $\phi\in \calT_{n,m}$, one has $\phi(\bfz_\bfiota;\bfh_\bfiota)\ne 0$.
\end{enumerate}
Here, we interpret the condition (iv) to be void when $m=0$. It follows that whenever
$(\bfz,\bfh)\in \calS_n$ with $1\le n<s$, then it is of type $T_{n,m}$ for some index
$m$ with $0\le m\le n$. As before, we introduce the notation $\calS_{n,m}$ to denote the set of all solutions of type $T_{n,m}$. We thus have the decomposition
\begin{equation}\label{5.2}
    \calS = \calS_0\cup \calS_s \cup \bigcup_{n=1}^{s-1} \bigcup_{m=0}^n \calS_{n,m}.
\end{equation}

\section{A divisor estimate}
Having enunciated our classification of solutions in the previous
section, we are equipped to estimate the number of solutions of the system (\ref{5.1})
with $1\le \bfz\le X$ and $|\bfh|\le X^r$. This will establish Theorem \ref{theorem2.1},
since by discarding superfluous equations if necessary, we may always suppose that
$t=2s-1$. Before embarking on the main argument, we establish a simple auxiliary result.

\begin{lemma}\label{lemma6.1} Suppose that $f \in \dbZ[x]$ is a polynomial of degree
$k\ge 1$. Let $u$ be an integer with $1\le u\le k$, and let $h_i$ and $a_i$ be fixed
integers for $1\le i\le u$ with $\bfh\ne {\mathbf 0}$ and $a_i\ne a_j$ $(1\le i<j\le u)$.
Then for any integer $n$, the equation
\begin{equation}\label{6.a}
        \sum_{i=1}^u h_i f(z+a_i) = n
\end{equation}
    has at most $k$ solutions in $z$.
\end{lemma}

\begin{proof}
It suffices to show that the polynomial in $z$ on the left hand side of (\ref{6.a}) has
positive degree. We therefore assume the opposite and seek a contradiction.
Suppose that $f$ is given by
    $$
        f(z) = c_k z^k + c_{k-1}z^{k-1} + \ldots + c_1 z + c_0,
    $$
    where $c_k \neq 0$. The polynomial on the left hand side of (\ref{6.a}) takes the
shape
    $$
        F(z) = d_k z^k + d_{k-1}z^{k-1} + \ldots + d_1 z + d_0,
    $$
    with
    $$
        d_i = \sum_{j=i}^k  c_j\binom{j}{i}(h_1a_1^{j-i}+\ldots +h_ua_u^{j-i})\quad
(0\le i\le k).
    $$
    In particular, we see directly that $d_k$ can vanish only if $h_1+\ldots +h_u=0$.
Let $i$ be a positive integer with $i<k$, and suppose that one has
    \begin{equation}\label{6.0}
        h_1a_1^{k-j}+\ldots +h_ua_u^{k-j}=0
    \end{equation}
for all integers $j$ with $i <j \le k$. Then the vanishing of $d_i$ implies that \eqref{6.0}
holds also for $j=i$. Proceeding inductively in this way, we deduce that \eqref{6.0} is
satisfied for the entire range $1\le j \le k$. Restricting attention to the system of
equations with indices $k-u+1 \le j \le k$, we find that this system of equations can hold
simultaneously only when either $\bfh={\mathbf 0}$, or else
    $$
        0=\det(a_i^{j-1})_{1\le i,j\le u}=\prod_{1\le i<j\le u}(a_i-a_j).
    $$
In the latter case, one has $a_i=a_j$ for some indices $i$ and $j$ with $1\le i<j\le u$.
Both these cases are excluded by the hypotheses of the statement of the lemma, so the
system of equations \eqref{6.0} cannot hold for all $1 \le j \le k$, and hence the
polynomial $F$ is non-trivial of positive degree. Consequently, the equation (\ref{6.a}) has
at most $\text{deg}(F)\le k$ solutions in $z$.
\end{proof}

\begin{proof}[The proof of Theorem \ref{theorem2.1}] We begin by examining the
solutions of (\ref{5.1}) of type $S_0$, recalling that $1\le \bfz\le X$ and $|\bfh|\le X^r$.
When $(\bfz,\bfh)\in \calS_0$, one has $h_if_1(z_i)=0$ for $1\le i\le 2s$. Suppose that
the indices $i$ for which $h_i=0$ are $i_1,\ldots,i_a$, and the indices $j$ for which
$h_j\ne 0$ are $j_1,\ldots ,j_b$. In
particular, one has $a+b=2s$. By relabelling variables, if necessary, there is no loss of
generality in supposing that $\bfj=(1,\ldots ,b)$ and $\bfi=(b+1,\ldots ,2s)$. There are
$O(X^{2s-b})$ possible choices for $h_i$ and $z_i$ with $b+1\le i\le 2s$, since $h_i=0$
for these indices $i$. Meanwhile, for $1\le j\le b$, one has $f_1(z_j)=0$, and so there are
at most $k_1$ possible choices for $z_j$. For each fixed such choice, since the
polynomials $f_1,\ldots ,f_t$ are well-conditioned, we find that $f_l(z_j)\ne 0$ for some
index $l$ with $2\le l\le t$. Thus, the variables $h_1,\ldots ,h_b$ satisfy a system of $t$
linear equations in which there are non-vanishing coefficients. We deduce that when
$b\ge 1$, there are $O((X^r)^{b-1})$ possible choices for $h_j$ and $z_j$ with
$1\le j\le b$. Finally, combining these estimates for all possible choices of $\bfi$ and
$\bfj$, we discern that
\begin{equation}\label{6.b}
    \card \calS_0\ll X^{2s}+\sum_{b=1}^{2s}X^{2s-b}\cdot X^{r(b-1)}\ll X^{(2s-1)r+1}.
\end{equation}
\par

Next we consider the solutions of (\ref{5.1}) of type $S_s$. When $(\bfz,\bfh) \in \calS_s$, there is an $s$-tuple $\bfi$ with $1\le i_1<\ldots <i_s\le 2s$ for which one has
$$
    \Psi_{s-1}(\sig_{1,s}(\bfz_\bfi;\bfh_\bfi),\ldots ,\sig_{2s-1,s}(\bfz_\bfi;\bfh_\bfi))\ne 0.
$$
Write $\bfi'$ for the $s$-tuple $(i'_1,\ldots ,i'_s)$ with $1\le i'_1<\ldots <i'_s\le 2s$ for
which
$$
    \{i_1,\ldots ,i_s\}\cup \{i'_1,\ldots ,i'_s\}=\{1,2,\ldots ,2s\}.
$$
It follows from (\ref{5.1}) that $\sig_{j,s}(\bfz_\bfi;\bfh_\bfi)=\sig_{j,s}(\bfz_{\bfi'};
-\bfh_{\bfi'})$ $(1\le j\le 2s-1)$, and hence there is a non-zero integer
$N=N(\bfz_{\bfi'};\bfh_{\bfi'})$ for which
\begin{equation}\label{6.1}
    \Psi_{s-1}(\sig_{1,s}(\bfz_{\bfi'};-\bfh_{\bfi'}),\ldots
    ,\sig_{2s-1,s}(\bfz_{\bfi'};-\bfh_{\bfi'}))=N
\end{equation}
and
$$
    \Psi_{s-1}(\sig_{1,s}(\bfz_\bfi;\bfh_\bfi),\ldots ,\sig_{2s-1,s}(\bfz_\bfi;\bfh_\bfi))=N.
$$
By relabelling variables, if necessary, there is no loss of generality in supposing that
$\bfi=(1,2,\ldots ,s)$ and $\bfi'=(s+1,s+2,\ldots ,2s)$.\par

Fix any one of the $O(X^{(r+1)s})$ possible choices for $\bfz_{\bfi'}$, $\bfh_{\bfi'}$
with $1\le \bfz_{\bfi'}\le X$, $|\bfh_{\bfi'}|\le X^r$, and satisfying (\ref{6.1}). Then we
infer from Lemma \ref{lemma3.2} that
\begin{equation}\label{6.2}
    h_1\cdots h_s\prod_{1\le i<j\le s}(z_i-z_j)\quad \text{divides}\quad
    N(\bfz_{\bfi'};\bfh_{\bfi'}).
\end{equation}
Moreover, one has $N(\bfz_{\bfi'};\bfh_{\bfi'})\ne 0$. Since the latter integer is fixed, we
see by means of an elementary divisor function estimate that there are $O(X^\eps)$
possible choices for $h_1,\ldots ,h_s$ and integers $a_2,\ldots ,a_s$ with the property
that $z_i=z_1+a_i$ $(2\le i\le s)$. With the exception of the undetermined variable $z_1$,
it follows that there are at most $O(X^{(r+1)s+\eps})$ possible choices for all the
variables in question. However, the integer $z_1$ satisfies the system of equations
\begin{equation}\label{6.3}
    h_1f_j(z_1)+\sum_{i=2}^sh_if_j(z_1+a_i)=n_j\quad (1\le j\le 2s-1),
\end{equation}
in which $h_i$, $a_i$ and $n_j$ are all fixed for all indices $i$ and $j$.  Consider the polynomial with index $j=1$ of largest degree $k_1\ge 2s-2$.
If $a_i$ is zero for any index $i$, then we have $z_1=z_i$. Meanwhile, if $a_i=a_j$ for any
indices $i$ and $j$ with $2\le i<j\le s$, one sees that $z_i=z_j$. Consequently, in either of
these scenarios, and also in the situation with $\bfh={\mathbf 0}$, one finds via
(\ref{6.2}) that $N(\bfz_{\bfi'};\bfh_{\bfi'})=0$, contradicting our assumption that
$N(\bfz_{\bfi'};\bfh_{\bfi'})\ne 0$.
We may thus safely assume that the conditions of Lemma \ref{lemma6.1} are satisfied
for the polynomial $f_1$ with $a_1=0$. By the conclusion of the lemma, it follows that
there are at most $k_1$ choices for $z_1$ satisfying (\ref{6.3}), and hence
\begin{equation}\label{6.3a}
    \card \calS_s\ll X^{(r+1)s+\eps}.
\end{equation}

\par Next we consider the set $\calS_{n,m}$  for a given pair of indices $n$ and $m$ with
$1\le n<s$ and $0\le m\le n$. For any $(\bfz,\bfh) \in \calS_{n,m}$, condition (ii) holds
for some $n$-tuple $\bfj$. By relabelling variables, if necessary, we may suppose that
$\bfj=(1,\ldots ,n)$. Write $\bfj'$ for the $(2s-n)$-tuple $(n+1,\ldots ,2s)$. Then given any one fixed choice of
the variables $\bfz_{\bfj'}$, $\bfh_{\bfj'}$, we have
\begin{align*}
    \Psi_{n-1}(\sig_{1,n}(\bfz_\bfj;\bfh_\bfj)&,\ldots ,\sig_{2n-1,n}(\bfz_\bfj;\bfh_\bfj))\\
    &=\Psi_{n-1}(\sig_{1,2s-n}(\bfz_{\bfj'};-\bfh_{\bfj'}),\ldots ,\sig_{2n-1,2s-n}
    (\bfz_{\bfj'};-\bfh_{\bfj'}))\ne 0.
\end{align*}
Thus, there is a fixed non-zero integer $N$ with the property that
$$
    \Psi_{n-1}(\sig_{1,n}(\bfz_\bfj;\bfh_\bfj),\ldots ,\sig_{2n-1,n}(\bfz_\bfj;\bfh_\bfj))=N,
$$
and we deduce from Lemma \ref{lemma3.2} that
$$
    h_1\cdots h_n\prod_{1\le i<j\le n}(z_i-z_j)\quad \text{divides}\quad N.
$$
From here, the argument applied above in the case $n=s$ may be employed mutatis
mutandis to conclude that there are $O(X^\eps)$ possible choices for
$h_1,\ldots ,h_n$, $z_1-z_2,\ldots ,z_1-z_n$. If we put $a_i=z_i-z_1$ $(2\le i\le n)$ and
$a_1=0$, then we find just as in our earlier analysis that $z_1$ satisfies a non-trivial
polynomial equation of degree at most $k_1$, whence there are at most $k_1$ choices
for $z_1$. We therefore conclude that, given any one fixed choice of
$\bfz_{\bfj'},\bfh_{\bfj'}$, the number of choices for $\bfz_\bfj,\bfh_\bfj$ is $O(X^\eps)$.
\par

It thus remains to count the number of choices for $\bfz_{\bfj'}$ and $\bfh_{\bfj'}$. Note
in particular that, since $(\bfz,\bfh) \in \calS_{n,m}$, we have the additional information
that conditions (iii) and (iv) are satisfied. We may therefore suppose that there exists
some
$\phi\in \calT_{n,m}$, and some $m$-tuple $(\iota_1,\ldots ,\iota_m)$ with
$n+1\le \iota_1<\ldots <\iota_m\le 2s$, for which
\begin{equation}\label{6.4a}
    \phi(\bfz_\bfiota;\bfh_\bfiota)\ne 0.
\end{equation}
With a fixed choice of $\bfiota$, we may suppose further that for all $i$ satisfying
$n+1\le i\le 2s$ and
$i\not\in \{\iota_1,\ldots ,\iota_m\}$, and for all $\psi\in \calT_{n,m+1}$, one has
\begin{equation}\label{6.5}
    \psi(z_{\iota_1},\ldots ,z_{\iota_m},z_i;h_{\iota_1},\ldots ,h_{\iota_m},h_i)=0.
\end{equation}
Given any such $\bfiota$ and $\phi$, there are $O(X^{(r+1)m})$ possible choices
for $\bfz_{\bfiota},\bfh_{\bfiota}$, with $1\le \bfz_\bfiota\le X$ and
$|\bfh_\bfiota|\le X^r$, satisfying \eqref{6.4a}. We claim that for any fixed such choice,
the number of possible choices for these integers $z_i$ and $h_i$ with $n+1\le i\le 2s$
and $i\not\in \{\iota_1,\ldots ,\iota_m\}$ is $O((X^r)^{2s-n-m})$. In order to confirm
this claim, observe that there is a polynomial $\psi\in \calT_{n,m+1}$ having the property
that some coefficient of $\psi(z_1,\ldots ,z_{m+1};h_1,\ldots ,h_{m+1})$, considered as a
polynomial in $z_{m+1}$ and $h_{m+1}$, is equal to
$\phi(z_1,\ldots ,z_m;h_1,\ldots ,h_m)$. It then follows from \eqref{6.4a} that the
equation \eqref{6.5} is a non-trivial polynomial equation in $z_i$ and $h_i$. We therefore
deduce from Lemma \ref{lemma4.2} that for each fixed choice of $\bfz_\bfiota$ and
$\bfh_\bfiota$ under consideration, and for each $i$ with $n+1\le i\le 2s$ and
$i\not \in \{\iota_1,\ldots ,\iota_m\}$, there are $O(X^r)$ possible choices for $z_i$ and
$h_i$ satisfying (\ref{6.5}). Thus we infer that there are $O(X^{r(2s-n-m)})$ possible
choices for $z_i$ and $h_i$ with $n+1\le i\le 2s$ for each fixed choice of
$\bfz_{\bfiota},\bfh_{\bfiota}$. Since the number of choices for $\bfiota$ and
$\phi \in \calT_{n,m}$ is $O(1)$, the total number of choices for $\bfz_{\bfj'}$ and
$\bfh_{\bfj'}$ available to us is $O(X^{(r+1)m}\cdot X^{r(2s-n-m)})$. Furthermore, our
discussion above shows that for each fixed such choice of $\bfz_{\bfj'}$, $\bfh_{\bfj'}$,
the number of possible choices for $\bfz_{\bfj},\bfh_{\bfj}$ is $O(X^\eps)$. Thus
altogether we conclude that
\begin{equation}\label{6.5a}
    \card \calS_{n,m}\ll X^{r(2s-n)+m+\eps}.
\end{equation}

\par
By combining our estimates \eqref{6.b}, \eqref{6.3a} and \eqref{6.5a} via \eqref{5.2},
we discern that
$$\card \calS \ll X^{(2s-1)r+1}+X^{(r+1)s+\eps} + \sum_{n=1}^{s-1}
\sum_{m=0}^n X^{r(2s-n)+m+\eps} \ll X^{r(2s-1)+1+\eps},
$$
and the conclusion of Theorem \ref{theorem2.1} follows.
\end{proof}

\section{The proof of Theorems \ref{theorem1.1} and \ref{theorem1.2}}
Our preparations now complete, we establish the mean value estimates recorded in
Theorems \ref{theorem1.1} and \ref{theorem1.2}. Let $X$ be a large positive number,
and suppose that $s$ and $k$ are natural numbers with $k\ge 3$ and
$1\le s\le (k^2-1)/2$. We define the exponential sum $\grg_r(\bfalp;X)$ by putting
\begin{equation}\label{7.1}
    \grg_r(\bfalp;X)=\sum_{|h|\le sX^r}\sum_{1\le z\le X}
    e\left(\binom{r}{r}h\alp_r+\binom{r+1}{r}hz\alp_{r+1}+\ldots +\binom{k}{r}h
    z^{k-r}\alp_k\right).
\end{equation}
Also, when $1\le d\le k$, we put
$$
    \grh_d(\bfalp;X)=\sum_{1\le x\le X}e(\alp_1x+\ldots +\alp_dx^d).
$$
Then, with the standard notation associated with Vinogradov's mean value theorem in mind,
we put
\begin{equation}\label{7.2}
    J_{\sig,d}(X)=\oint |\grh_d(\bfalp;X)|^{2\sig}\d\bfalp .
\end{equation}
We note that the main conjecture in Vinogradov's mean value theorem is now known to
hold for all degrees. This is a consequence of work of the second author for degree $3$,
and for degrees exceeding $3$ it follows from the work of Bourgain, Demeter and Guth
(see \cite[Theorem 1.1]{Woo2016} and \cite[Theorem 1.1]{BDG2016}). Thus, one has
    \begin{equation}\label{7.5}
        J_{\sig,d}(X)\ll X^{\sig+\eps}\quad (1\le \sig\le d(d+1)/2).
    \end{equation}
    In addition, one finds via orthogonality that for each integer $\kap$, one has
    $$
        \oint|\grg_r(\bfalp;X)|^{2\kap}\d\bfalp \le A_{\kap,r}(sX;\bff),
    $$
where $f_j(z)=z^{k-r+1-j}$ $(1\le j\le k-r+1)$.

\begin{lemma}\label{lemma7.1}
    When $s$ is a natural number, one has
    $$
        I_{s,k,r}(X)\ll X^{-1}\oint |\grh_k(\bfalp;2X)|^{2s}\grg_r(-\bfalp;X)\d\bfalp .
    $$
\end{lemma}

\begin{proof}
Define $\del_j$ to be $1$ when $j=r$, and $0$ otherwise. We start
by noting that the mean value $I_{s,k,r}(X)$ counts the number of integral solutions of the
system of equations
\begin{equation}\label{7.3}
    \sum_{i=1}^s(x_i^j-y_i^j)=\del_jh\quad (1\le j\le k),
\end{equation}
with $1\le x_i,y_i\le X$ $(1\le i\le s)$ and $|h|\le sX^r$. We remark that the constraint on
\begin{equation}\label{7.4}
    \sum_{i=1}^s(x_i^r-y_i^r)
\end{equation}
imposed by the equation of degree $r$ in (\ref{7.3}) is void, since the range for $h$
automatically accommodates all possible values of the expression (\ref{7.4}) within
(\ref{7.3}).\par

We next consider the effect of shifting every variable by an integer $z$ with $1\le z\le X$.
By the binomial theorem, for any shift $z$, one finds that $(\bfx, \bfy)$ is a solution of
(\ref{7.3}) if and only if it is also a solution of the system
$$
    \sum_{i=1}^s\left( (x_i+z)^j-(y_i+z)^j\right)=\ome_jhz^{j-r}\quad (1\le j\le k),
$$
where $\ome_j$ is $0$ for $1\le j<r$ and $\binom{j}{r}$ for $r\le j\le k$. Thus, for each
fixed integer $z$ with $1\le z\le X$, the mean value $I_{s,k,r}(X)$ is bounded above by the
number of integral solutions of the system
$$
    \sum_{i=1}^s(u_i^j-v_i^j)=\ome_jhz^{j-r}\quad (1\le j\le k),
$$
with $1\le \bfu,\bfv\le 2X$ and $|h|\le sX^r$. On applying orthogonality, we therefore
infer that
$$
    I_{s,k,r}(X)\ll X^{-1}\sum_{1\le z\le X}\oint |\grh_k(\bfalp;2X)|^{2s}
\grf(-\bfalp; z)\d\bfalp ,
$$
where
$$
    \grf(\bfalp;z)=\sum_{|h|\le sX^r}e\left(\ome_r h\alp_r+
    \ome_{r+1}hz\alp_{r+1}+\ldots +\ome_k hz^{k-r}\alp_k\right).
$$
The proof of the lemma is completed by reference to (\ref{7.1}).
\end{proof}

\begin{proof}[The proof of Theorem \ref{theorem1.2}]
    Let $s$, $k$ and $r$ be integers
    with $k>r\ge 1$. Let $\kap$ be a positive integer with $\kap\le (k-r+2)/2$, put  $$s=\left\lfloor\frac{k(k+1)}{2}-\frac{k(k+1)-r(r-1)}{4\kap}\right\rfloor,
    $$
    and then let
    $$
        v=\frac{r(r-1)}{4\kap}\quad \text{and}\quad u=s-v.
    $$
    Furthermore, set
    $$
        w=\left(1-\frac{1}{2\kap}\right)\frac{k(k+1)}{2},
    $$
    so that $s=\lfloor v+w \rfloor$. In particular, we have $w\ge u$.\par

On applying H\"older's inequality in combination with Lemma \ref{lemma7.1}, we find
    that
    \begin{equation}\label{7.7}
        I_{s,k,r}(X)\ll X^{-1}U_1^{1-1/(2\kap)}U_2^{1/(2\kap)},
    \end{equation}
    where
    \begin{equation}\label{7.8}
        U_1=\oint |\grh_k(\bfalp;2X)|^{(u/w)k(k+1)}\d\bfalp
    \end{equation}
    and
    \begin{equation}\label{7.9}
        U_2=\oint |\grh_k(\bfalp;2X)^{r(r-1)}\grg_r(\bfalp;X)^{2\kap}|\d\bfalp.
    \end{equation}

    A comparison of (\ref{7.8}) with (\ref{7.2}) leads us via (\ref{7.5}) to the estimate
    \begin{equation}\label{7.10}
        U_1\ll X^{(u/w)k(k+1)/2+\eps}.
    \end{equation}
Meanwhile, by orthogonality, we discern from (\ref{7.9}) that $U_2$ counts the number of
integral solutions of the system of equations
    \begin{align}
        \sum_{i=1}^{r(r-1)/2}(x_i^j-y_i^j)&=\binom{j}{r}\sum_{l=1}^{2\kap}h_lz_l^{j-r}&
        (r\le j\le k)\label{7.11}\\
        \sum_{i=1}^{r(r-1)/2}(x_i^j-y_i^j)&=0&(1\le j<r),\label{7.12}
    \end{align}
with $1\le \bfx,\bfy\le 2X$, $1\le \bfz\le X$ and $|\bfh|\le sX^r$. By interpreting
(\ref{7.12}) through the prism of orthogonality, it follows from (\ref{7.2}) that the number
of available choices for $\bfx$ and $\bfy$ is bounded above by $J_{r(r-1)/2,r-1}(2X)$. For
each fixed such choice of $\bfx$ and $\bfy$, it follows from (\ref{7.11}) via orthogonality
and the triangle inequality that the number of available choices for $\bfz$ and $\bfh$ is at
most $A_{\kap,r}(sX;\bff)$. Thus we deduce from (\ref{7.5}) and Theorem
\ref{theorem2.1} that
    \begin{equation}\label{7.13}
        U_2\le J_{r(r-1)/2,r-1}(2X)A_{\kap,r}(sX;\bff)\ll X^{r(r-1)/2+r(2\kap-1)+1+\eps}.
    \end{equation}
    On substituting (\ref{7.10}) and (\ref{7.13}) into (\ref{7.7}), we infer that
    $$
        I_{s,k,r}(X)\ll X^{\eps-1} (X^{(u/w)k(k+1)/2})^{1-1/(2\kap)}(X^{2r\kap+1+r(r-3)/2})^{1/(2\kap)}\ll X^{s+\Del+\eps},
    $$
    where
    $$
        \Del=(r-1)-\frac{r-1}{2\kap}.
    $$
    This completes the proof of Theorem \ref{theorem1.2}.
\end{proof}

\begin{proof}[The proof of Theorem \ref{theorem1.1}]
    The conclusion of Theorem
    \ref{theorem1.1} is an immediate consequence of Theorem \ref{theorem1.2} in the special
    case $r=1$. Making use of the notation of the statement of the latter theorem, we note
    that when $k=2l+1$ is odd, one may take $\kap=\lfloor (k+1)/2\rfloor =l+1$, and we deduce that
    $I_{s,k,1}(X)\ll X^{s+\eps}$ provided that $s$ is a natural number not exceeding
    $$
        \frac{k(k+1)}{2}-\frac{k(k+1)}{4(l+1)}=\frac{k(k+1)}{2}-\frac{k}{2}.
    $$
    Meanwhile, when $k=2l$ is even, one may instead take $\kap=l$, and the same conclusion holds
    provided that $s$ is a natural number not exceeding
    $$
        \frac{k(k+1)}{2}-\frac{k(k+1)}{4l}=\frac{k(k+1)}{2}-\frac{k+1}{2}.
    $$
    The desired conclusion therefore follows in both cases, and the proof of Theorem
    \ref{theorem1.1} is complete.
\end{proof}

\bibliographystyle{amsbracket}
\providecommand{\bysame}{\leavevmode\hbox to3em{\hrulefill}\thinspace}

\end{document}